\documentclass[11pt]{article}
\usepackage{latexsym,amsmath,amsthm,verbatim,ifthen,amssymb,graphicx,color}
\usepackage[all]{xy}

\SelectTips{cm}{}

\newtheorem{theorem}{Theorem}[section]

 \theoremstyle{definition}
 \newtheorem{definition}[theorem]{Definition}
 \theoremstyle{remark}
 \newtheorem{remark}[theorem]{Remark}
 
 \numberwithin{equation}{subsection}
\newcommand{\mc}{\operatorname{{\rm MC}}}
 \newcommand{\ad}{\operatorname{{\rm ad}}}
\newcommand{\bz}{\mathbb Z}
\newcommand{\bn}{\mathbb N}
\newcommand{\bk}{\mathbb K}
\newcommand{\bq}{\mathbb Q}

 \newcommand{\lib }{\mathbb{L}}

    \newcommand{\lasu}{{\mathfrak{L}}}
    
\newcommand{\aes}{\alpha }
\newcommand{\bes}{\beta }
\newcommand{\ees}{x }

\newcommand{\bchl }{\sum_{k=0}^\infty \sum_{p+q=k}(-1)^q\frac{B_{p+q}}{p!q!}\ees^p(\bes-\aes)\ees^q}

\begin{document}
\title{The gauge action, DG Lie algebras and identities for Bernoulli numbers}


\author{Urtzi Buijs\footnote{Partially supported by the \emph{Ministerio de Econom\'{\i}a y Competitividad} grants MTM2010-15831, MTM2013-41768-P, by the  grants FQM-213, 2009-SGR-119, and by the Marie Curie COFUND programme U-mobility, co-financed by the University of M\'alaga, the European Commision FP7 under GA No. 246550, and Ministerio de Econom\'{\i}a y Competitividad (COFUND2013-40259).},\hskip2mm J. G. Carrasquel-Vera\footnote{Partially supported by the  \emph{Ministerio de Econom\'{\i}a y Competitividad} grant MTM2010-18089.}  \ and Aniceto Murillo \footnote{Partially supported by the  \emph{Ministerio de Econom\'{\i}a y Competitividad} grant MTM2013-41768-P and by the Junta de
Andaluc\'\i a grants FQM-213 and P07-FQM-2863. \vskip
 4pt
 Key words and phrases: Gauge action. Bernoulli numbers. Homotopy theory of Lie algebras.}}

\maketitle

\begin{abstract}
In this paper we prove a family of identities for Bernoulli numbers parameterized by triples of integers $(a,b,c)$ with $a+b+c=n-1$, $n\ge 4$. These identities are deduced by  translating into homotopical terms the gauge action on the Maurer Cartan set of a differential graded Lie algebra.
We show that  Euler and Miki's identities, well known and apparently non related formulas,
 are  linear combinations of our family and they satisfy a particular symmetry relation.
\end{abstract}

\section*{Introduction}
In our setting, an {\em Euler-type identity} is defined to be a convolution equation of the form
$$
\sum_{\substack{k=0\\ k\,\,\text{even}}}^n \lambda _{k} B_kB_{n-k}=0, \quad\text{for any even $n\ge 4$,}
$$
where $\lambda_{k}$ are rational numbers depending on $k$ and $n$, and $B_k$ are the Bernoulli numbers for which \cite{libro} is an excellent general reference. The Euler equation,
\begin{equation}\label{euler}
-(n+1)B_n=\sum_{\substack{k=2\\ k\,\,\text{even}}}^{n-2}\binom{n}{k}B_kB_{n-k},\ n \ \text{{even with}}\ n\geq 4,
\end{equation}
as well as the, now classical, Miki's  identity
 \cite{miki},

\begin{equation}\label{miki}
2H_n{B_n}=\sum_{\substack{k=2\\ k\,\,\text{even}}}^{n-2}\frac{n}{k(n-k)}\Bigr( 1-\binom{n}{k}\Bigl){B_k}{B_{n-k}},\ n \ \text{{even with}}\ n\geq 4,
\end{equation}
where $H_n=\sum _{j=1}^n \frac{1}{j}$ is the harmonic number, are particular instances of Euler-type identities.


Combinatorial, arithmetical, analytical and geometrical methods have been used to deduced Euler-type identities \cite{crabb,dunne,fapan,gessel,pan2,patan}. Some of them, including Euler and Miki's,  are special cases of a large family of Euler-type identities whose existence we prove in this paper. Namely,

\begin{theorem}\label{main}
For any even integer $n\geq 4$ and any triple of {non negative} integers $(a,b,c)$ such that $a+b+c=n-1$,
 an Euler-type identity holds with
\begin{equation}\label{mainequations}
\lambda_{k}={\lambda_{k}^{(a,b,c)}}=\binom{n}{k}\Bigg[(-1)^c\binom{n-k}{c}\sum_{\ell=\text{max\,}(0, k-b)}^{\text{min\,}(a,k)}(-1)^{{k-}\ell }\binom{k}{\ell }
\end{equation}
\begin{equation*}
\ \ \ \ \ \ -(-1)^a\binom{n-k}{a}\sum_{\ell=\text{max\,}(0, k-b)}^{\text{min\,}(c,k)}(-1)^{\ell }\binom{k}{\ell }\Bigg].
\end{equation*}
\end{theorem}
We now briefly describe the methods we use to prove this result.

A fundamental principle of deformation theory due to Deligne asserts that every
deformation functor is governed by a differential graded Lie algebra  via solutions
of Maurer-Cartan equation modulo the gauge action. This permits us to view deformations of structures in completely different settings under a common algebraic scope. The quantization theorem of Kontsevich \cite{kont} or the deformation theory of Floer homology of lagrangian submanifolds and its relation with mirror symmetry \cite{fu} are good examples of this.

Given a differential graded Lie algebra, the gauge action on its Maurer-Cartan set, which may be understood as an algebraic abstraction of the behavior of gauge infinitesimal transformations in classical gauge theory, can be encoded via  the {\em Lawrence-Sullivan construction} $\lasu$, see \cite[Prop.3.1]{bumu1} or \cite[4.6]{bumu2}. Then, we show that Theorem \ref{main} above is equivalent to Theorem 1 of \cite{lasu}, that is,   $\lasu$ is indeed a differential graded Lie algebra. This is done by transporting $\lasu$ to the category of differential graded algebras via the universal enveloping functor and forcing it to be a cylinder in the corresponding homotopy category \cite[Thm. 3.3]{bumu1}.

Then, we show that the gauge action, the existence of such a cylinder in the homotopy category of differential graded algebras, and the Euler-type identities in Theorem \ref{main}, are equivalent formulations of the same statement.

In the next section, devoted to the proof of our main result and its condensed version, we will develop in detail the above sketch. In Section 2, see Theorem \ref{thmmiki}, we show that Miki and Euler's  identities are particular instances of our family. Moreover, in  Theorem \ref{thmeuler}, we find an unexpected symmetry relation between them.

\section{Euler-type identities and the gauge action}

This section, entirely devoted to the proof of Theorem \ref{main}, begins with a brief but explicit description of the algebraic version of the gauge action. From now on, any considered algebraic object is assumed to be $\bz$-graded and over a coefficient field $\bk$ of characteristic zero.

Recall that a {\em differential graded Lie algebra}, DGL henceforth, is a graded vector space $L=\oplus_{n\in\bz}L_n$ endowed of a bilinear bracket of degree zero,
$$
[\,\,,\,]\colon L\otimes L\to L,
$$
and a differential  $\partial\colon L\to L$ of degree $-1$ satisfying,
\begin{eqnarray*}
[x,y]&=&(-1)^{|x||y|+1}[y,x],\\
\bigl[x,[y,z]\bigr]&=&\bigl[[x,y],z\bigr]+(-1)^{|x||y|}\bigl[y,[x,z]\bigr],\\
\partial[x,y]&=&[\partial x,y]+(-1)^{|x|}[x,\partial y],
\end{eqnarray*}
for any homogeneous elements $x,y,z\in L$.

 The {\em Maurer-Cartan set} $\mc(L)$ of $L$ is formed by those elements $a\in L_{-1}$ whose differential satisfy the {\em Maurer-Cartan equation},
$$
\partial \alpha+\frac{1}{2}[\alpha,\alpha]=0.
$$
The \emph{completion} $\widehat{L}$ of a graded Lie algebra $L$ is the limit
$$\widehat{L} = \varinjlim_n L/L^n$$
where $L^1 = L$ and for $n\geq 2$, $L^n = [L, L^{n-1}]$. A Lie algebra $L$ is  \emph{complete} if it is isomorphic to its completion. For such a Lie algebra, the {\em gauge action} of $L_0$ on $\mc(L)$ is defined as
$$
x*\alpha=\sum_{i\ge0}\frac{\ad_x^i(\alpha)}{i!}-\sum_{i\ge0}\frac{\ad_x^{i}(\partial x)}{(i+1)!},\quad x\in L_0,\quad \alpha\in\mc(L).
$$
 This can be geometrically interpreted as follows \cite{lasu}: in the DGL given by  $L[t]=L\otimes\bk[t]$ consider the formal differential equation
 $$
\begin{aligned}
&u'(t)=\partial x -\ad_xu(t),\\
&u(0)=\alpha.
\end{aligned}
$$
Then, writing $u(t)$ as a formal power series, one has $x*\alpha=u(1)$. In other words, thinking of $\mc(L)$ as points (say of a ``formal manifold"), the gauge action can be thought of as the flow at time $1$, generated by $x$ via the above equation, with initial point $\alpha$.

The minimal algebraic expression of this  was given in \cite{lasu}: Consider
$$\lasu=\widehat\lib(\alpha,\beta,x)
$$
the {\em complete free Lie algebra}, that is, the completion of the free Lie algebra  generated by the Maurer-Cartan elements $\alpha,\beta$ and in which the flow generated by $x$ moves from $\alpha$ to $\beta$, i.e., $x*\alpha=\beta$. Then,

\begin{theorem}\label{lasu1}{\em\cite[Thm. 1.4]{bufemutan},\cite[Thm.1]{lasu}} The  unique choice for $\partial x$ which makes $\lasu$ a DGL is
$$
\partial x=\ad_x\beta+\sum_{i\ge0}\frac{B_i}{i!}\ad^i_x(\beta-\alpha).
$$
\end{theorem}
Remark that the gauge action in any differential graded Lie algebra $L$ is ``controlled" by $\lasu$. Indeed, see \cite[Prop.3.1]{bumu1} or \cite[4.6]{bumu2}, any two elements $a,b\in\mc(L)$ are gauge equivalent if and only if there is a DGL morphism $\lasu\to L$ sending $\alpha$ to $a$ and $\beta$ to $b$. In homotopical terms, $\lasu$ is a {\em universal cylinder} for the gauge relation, which is in turn equivalent to the classical Quillen homotopy notion for DGL's.

Now, observe that the injection $\lib(V)\hookrightarrow T(V)$ of the free Lie algebra generated by  $V$   into the tensor algebra induces an injection
$$
\widehat\lib(V)=\varinjlim_n \lib(V)/ \lib^{\ge n}(V)\hookrightarrow \varinjlim_n T(V)/T^{\ge n}(V)=\widehat T(V)
$$
into the {\em complete tensor algebra}. In particular,
$$
\lasu\hookrightarrow \widehat T(\alpha,\beta,x).
$$
In view of the following result, this permits an equivalent formulation of Theorem \ref{lasu1}, and thus of the gauge action, in the category of associative differential graded algebras.
\begin{theorem}\label{lasu2}{\em\cite[Thm.3.3]{bumu1}} The derivation   given by $D\alpha=-\alpha\otimes\alpha$, $D\beta=-\beta\otimes\beta$ and
$$
Dx=x\otimes\beta-\beta\otimes x+\sum_{k\ge0}\sum_{p+q=k}(-1)^q\frac{B_{p+q}}{p!q!}x^{\otimes p}\otimes(\beta-\alpha)\otimes x^{\otimes q},
$$ makes  $\widehat T(\alpha,\beta,x)$ a differential graded  algebra for which the above injection is a differential map.
\end{theorem}

We now show that  the three Theorems \ref{main}, \ref{lasu1} and \ref{lasu2} are all equivalent. Indeed, we will devote the rest of the section to show that the equation
$$
D^2=0$$ in Theorem \ref{lasu2}
is equivalent to the Euler-type identities of Theorem \ref{main}.

For simplicity on the notation we will drop  the $\otimes$ sign henceforth, so that
$
D\alpha=-\alpha^2$, $D\beta=-\beta^2$ and
\begin{equation}\label{formula1}
Dx=x\beta-\beta x+\bchl.
\end{equation}
We will also write
$$
y=\beta-\alpha,\qquad c_{(p,q)}=(-1)^q\frac{B_{p+q}}{p!q!},\qquad
\Phi=\sum_{k=0}^\infty\sum_{p+q=k}c_{(p,q)}x^pyx^q.
$$
Observe that, while  $D^2\alpha=D^2\beta=0$ holds trivially, a short computation shows that $D^2x=0$ is equivalent to
\begin{equation}\label{formula2}
D\Phi=-(\Phi\otimes\beta+\beta\otimes\Phi).
\end{equation}
Now, as $D$ is a derivation, it is easy to deduce from (\ref{formula1}) that
$$Dx^m=x^m\bes -\bes x^m+\sum_{i+j=m-1} x^i\Phi x^j.$$
Then, we have:\let\thefootnote\relax\footnote{$^{(*)}$A straightforward computation shows that $x^p(\bes y+Dy+y\bes )x^q=x^py^2x^q$.}
\begin{eqnarray*}
D(x^pyx^q)
&=&\Big( x^p\bes -\bes x^p+ \sum_{i+j=p-1}x^i {\Phi} x^j  \Big) yx^q+x^p(Dy)x^q\\
&&-x^py \Bigl( x^q\bes -\bes x^q+
\sum_{i+j=q-1}x^i {\Phi} x^j \Bigr)\\
&\stackrel{(*)}{=}&-(\bes x^pyx^q+x^pyx^q\bes)+ \Gamma_{(p,q)},
\end{eqnarray*}
where
\begin{equation}\label{gamma}
\Gamma_{(p,q)}=x^py^2x^q+\Bigl(\sum_{i+j=p-1}x^i {\Phi} x^j\Bigr)yx^q-x^py \Bigl(\sum_{i+j=q-1}x^i {\Phi} x^j \Bigr).
\end{equation}


Therefore,
\begin{eqnarray*}
D{\Phi} &=&\sum_{k=0}^\infty \sum_{p+q=k}c_{(p,q)}D(x^pyx^q)\\
&=&\sum_{k=0}^\infty \sum_{p+q=k}c_{(p,q)}\Bigl( -(\bes x^pyx^q+x^pyx^q\bes)+ \Gamma_{(p,q)}\Bigr)\\
&=&-\Bigl( {\Phi} \bes+\bes{\Phi} \Bigr)+\sum_{k=0}^\infty \sum_{p+q=k}c_{(p,q)}\Gamma_{(p,q)}.
\end{eqnarray*}
Hence, at the sight of (\ref{formula2}), we conclude that $D^2=0$ is equivalent to
\begin{equation}\label{gammas}
\sum_{k=0}^\infty \sum_{p+q=k}c_{(p,q)}\Gamma_{(p,q)}=0.
\end{equation}
Observe that the left hand side of the above equation can be rewritten as,
$$
\sum_{k=0}^\infty \sum_{p+q=k}c_{(p,q)}\Gamma_{(p,q)}=\sum_{n\ge1}\sum_{a+b+c=n-1}\gamma_{(a,b,c)}x^ayx^byx^c.
$$
 Thus, $D^2=0$ if and only if,  {for any triple of non negative integers $(a,b,c)$, we have}
$$
 \gamma_{(a,b,c)}=0.
$$
  Hence, Theorem \ref{main} will be established if we show that, for any even $n\ge 4$,
  $$
  {n!}\gamma_{(a,b,c)}=\sum_{\substack{k=0\\ k\,\,\text{even}}}^n \lambda _{k} B_kB_{n-k},\ \ {a+b+c=n-1,}
  $$
  with $\lambda_{k}$ as in (\ref{mainequations}).
For it, we find all the terms of the left hand side of (\ref{gammas}) contributing to the monomial  $x^ayx^byx^c$ for a fixed $(a,b,c)$.

Substituting formula $(\ref{gamma})$ and the definition of $\Phi $ into the left hand side of $(\ref{gammas})$, let us write $\gamma_{(a,b,c)}$ as
$$
\gamma_{(a,b,c)}=c_{(a,c)}\delta_{b}^0+\sum_{i=0}^a\sum_{j=0}^bc_{(i+j+1,c)}c_{(a-i,b-j)}-\sum_{i=0}^b\sum_{j=0}^cc_{(a, i+j+1)}c_{(b-i, c-j)},
$$
 where the first term comes from the term $x^py^2x^q$ in formula $(\ref{gamma})$. Here, $\delta$ stands for the usual Kronecker's delta.

Now, writing $k=a+b-i-j$ and $\ell=a-i$ in the first summation, and $k=b+c-i-j$ and $\ell=c-j$ in the second, we get,
$$(-1)^c\frac{B_{a+c}}{a!c!}\delta_{b}^0+\sum_{k=0}^{n}B_kB_{n-k}\frac{(-1)^c}{c!(n-c-k)!}\Bigr( \sum_{\ell=\text{max}(0, k-b)}^{\text{min}(a,k)}\frac{(-1)^{k-\ell}}{\ell !(k-\ell)!}\Bigl)$$
\begin{equation}\label{formula4}
-\sum_{k=0}^{n}B_{k}B_{n-k}\frac{(-1)^{n-k-a}}{a!(n-k-a)!}\Bigr(\sum_{\ell=\text{max}(0, k-b)}^{\text{min}(c,k)}\frac{(-1)^{\ell}}{\ell!(k-\ell)!}\Bigl)=0.
\end{equation}
Note that, the summations, which should be  $\sum_{k=0}^{a+b}$ and $\sum_{k=0}^{b+c}$, have been replaced by  $\sum_{k=0}^n$ (recall that $a+b+c=n-1$). Indeed, for the first summand, if $k\geq a+b+1$  then, $\text{min}(a, k)=a$ and $\text{max}(0, k-b)>a$. For the second summand, if $k\geq b+c+1$ then, $\text{min}(c, k)=c$ and $\text{max}(0, k-b)>c$.

{A short computation shows that ,whenever $b=0$, the terms in $B_1B_{n-1}$ equal $-(-1)^c\frac{B_{a+c}}{a!c!}$ and then they cancel with the first term.}

{Note also that, since odd Bernoulli numbers different from $B_1$ vanish, identity $(\ref{formula4})$ is tautologous unless $n$ is even${}^{\dagger}$\footnote{${}^{\dagger}$ In fact, this holds for $n$ even with $n\geq 4$, since for $n=2$ we obtain an identity of the form $\gamma_0B_0B_2+\gamma_1B_1B_1+\gamma_2B_2B_0=0$, which is not an Euler-type identity in our sense.} and it can be written as a sum over $k$ even, without the term $(-1)^c\frac{B_{a+c}}{a!c!}\delta_{b}^0$.}


Finally, multiplying each term by $n!$,  equation (\ref{formula4}) reduces to
$$\sum_{\substack{k=0\\ k\,\,\text{even}}}^n\lambda_{k} B_kB_{n-k}=0$$
with
\begin{eqnarray*}
\lambda_{k}=\binom{n}{k}\Bigg[&(-1)^c&\binom{n-k}{c}\sum_{\ell=\text{max\,}(0, k-b)}^{\text{min\,}(a,k)}(-1)^{k-\ell }\binom{k}{\ell }\\
-&(-1)^a&\binom{n-k}{a}\sum_{\ell=\text{max\,}(0, k-b)}^{\text{min\,}(c,k)}(-1)^{{\ell }}\binom{k}{\ell }\Bigg],
\end{eqnarray*}
and Theorem \ref{main} is proved.

We finish  by giving a ``condensed'' version of our main result which will be used in the next section. For it, as Euler-type equations can be thought of  as homogeneous equations on the Bernoulli numbers, they can be simplified as follows.
\begin{definition}

Given an Euler-type identity
$$\displaystyle\sum_{\substack{k=0\\ k\,\,\text{even}}}^n \lambda _k B_kB_{n-k}=0,
$$
its {\em condensed form} is defined as
{
$$\sum_{\substack{k=0\\ k\,\,\text{even}}}^{\frac{n}{2}-1}\mu_k B_kB_{n-k}+\frac{1}{2}\mu_{\frac{n}{2}}B_{\frac{n}{2}}^2=0,$$}
{where $\mu_k=\lambda_k+\lambda_{n-k}$, for $k$ even, $k\leq \frac{n}{2}$.}
\end{definition}

Then, the condensed form of Theorem \ref{main} reads:

\begin{theorem}\label{mainc}
For any even integer $n\geq 4$ and any triple of integers $(a,b,c)$ such that $a+b+c=n-1$,
there is a condensed Euler-type identity
$$
\sum_{\substack{k=0\\ k\,\,\text{even}}}^{\frac{n}{2}-1}\mu_k B_kB_{n-k}+\frac{1}{2}\mu_{\frac{n}{2}}B_{\frac{n}{2}}^2=0,
$$
in which
\begin{equation*}
\mu_k=\binom{n}{k}\Bigg[(-1)^c\binom{n-k}{c}\sum_{\ell=\text{max\,}(0, k-b)}^{\text{min\,}(a,k)}(-1)^{{k-}\ell }\binom{k}{\ell }
\end{equation*}
\begin{equation*}
+(-1)^c\binom{k}{c}\sum_{\ell=\text{max\,}(0,n-k-b)}^{\text{min\,}(a,n-k)}(-1)^{{n-k-}\ell }\binom{n-k}{\ell }
\end{equation*}
\begin{equation*}
-(-1)^a\binom{n-k}{a}\sum_{\ell=\text{max\,}(0, k-b)}^{\text{min\,}(c,k)}(-1)^{\ell }\binom{k}{\ell }
\end{equation*}
\begin{equation}\label{dobladas}
-(-1)^a\binom{k}{a}\sum_{\ell=\text{max\,}(0, n-k-b)}^{\text{min\,}(c,n-k)}(-1)^{\ell }\binom{n-k}{\ell }\Bigg].
\end{equation}
\end{theorem}

\section{Miki and Euler's identities}

For any even natural number $n\ge 4$, denote by $\Pi_{n-1}$  the ``natural valued" plane $x+y+z=n-1$, that is, $\Pi_{n-1}=\{(a,b,c)\in\bn^3,\,\,a+b+c=n-1\}$. Observe that Theorem \ref{main} defines a map
$$
f\colon \Pi_{n-1}\longrightarrow \bq^{{\frac{n}{2}+1}},\qquad f(a,b,c)=(\lambda_0,{\lambda_2},\dots,\lambda_n),
$$
whose image is contained in the $\bq$-vector space of solutions of the Euler-type equation
$$\sum_{\substack{k=0\\ k\,\,\text{even}}}^n\lambda_k B_kB_{n-k}=0.$$

It is immediate to observe that the classical Euler's equation (\ref{euler}) corresponds to $f(n-1,0,0)$. Indeed, applying directly Theorem \ref{main}, an easy computation yields,
\begin{eqnarray*}
\lambda_0&=&n+1,\qquad\lambda_n=0,\\
\lambda_k&=&\binom{n}{k},\qquad 1\le k< n, \ {k\ \text{even}.}
\end{eqnarray*}

In the same way,  the condensed version in Theorem \ref{mainc} provides a map
$$
g\colon \Pi_{n-1}\longrightarrow \bq^{{\lfloor\frac{n}{4}\rfloor}+1},\ \ g(a,b,c)={\left\{ \begin{array}{lcc}
             (\mu_0,\mu_2,\dots ,\mu_{\frac{n}{2}-1}) & \ \text{if}\ \frac{n}{2}\ \text{is odd,} \\
             (\mu_0,\mu_2,\dots ,\mu_{\frac{n}{2}-2}, \frac{1}{2}\mu_{\frac{n}{2}}) & \ \text{if}\ \frac{n}{2}\ \text{is even,}
             \end{array}
   \right. }
$$
({where $\lfloor\frac{n}{4}\rfloor$ denotes the integer part of $\frac{n}{4}$)}, whose image is contained in the $\bq$-vector space of solutions of the condensed Euler-type equation
$$
\sum_{\substack{k=0\\ k\,\,\text{even}}}^{\frac{n}{2}-1}\mu_k B_kB_{n-k}+\frac{1}{2}\mu_{\frac{n}{2}}B_{\frac{n}{2}}^2=0.
$$We will prove that the condensed form of Miki's identity (\ref{miki}), which is
$$
\sum_{\substack{k=0\\ k\,\,\text{even}}}^{\frac{n}{2}-1}M_k B_kB_{n-k}+\frac{1}{2}M_{\frac{n}{2}}B_{\frac{n}{2}}^2=0,
$$
where,
$$
M_0=-H_n,\qquad
M_k=\frac{n}{(n-k)k}\Bigl( 1-\binom{n}{k}\Bigr),\quad 1\leq k\le \frac{n}{2},\ {k\ \text{even},}
$$
is  generated by $\text{Im}\, g$.

Write
$$M={\left\{ \begin{array}{lcc}
             (M_0,M_2,\dots ,M_{\frac{n}{2}-1}) & \ \text{if}\ \frac{n}{2}\ \text{is odd,} \\
             (M_0,M_2,\dots ,M_{\frac{n}{2}-2}, \frac{1}{2}M_{\frac{n}{2}}) & \ \text{if}\ \frac{n}{2}\ \text{is even,}
             \end{array}
   \right. }$$
and   denote $g(0,n-j-1,j)$ simply by $g(j)$. Then, we show that $M$ is a $\bq$-linear combination of the vectors $g(j)$, $0\le j\le \frac{n}{2}$. Explicitly:

\begin{theorem} \label{thmmiki} For any even integer $n\ge4$,
$$
M=\sum_{j=0}^{\frac{n}{2}-1}\Bigl( \frac{1}{j+1}g(j)+\frac{1}{n-j}g(j+1)\Bigr).
$$
\end{theorem}

\begin{proof}
Denote the $k$th component of $g(k)$ by $g(j)_k$ .

Applying directly formula (\ref{dobladas}) in Theorem \ref{mainc} we obtain that,  for $0\le j\le n-1$,
$$
g(j)_0=(-1)^j\binom{n}{j}-1.
$$
Then, taking into account that $\frac{1}{j+1}\binom{n}{j}=\frac{1}{n-j}\binom{n}{j+1}$, a short computation shows  that
$$
\sum_{j=0}^{\frac{n}{2}-1}\Bigl( \frac{1}{j+1}g(j)_0+\frac{1}{n-j}g(j+1)_0\Bigr)=-\Bigl(1+\frac{1}{2}+\cdots+\frac{1}{n}\Bigr)=-H_n=M_0.
$$

On the other hand, for  any even $0<k\le {\frac{n}{2}}$, it is also easy to check that equation (\ref{dobladas}) in Theorem \ref{mainc} translates to:

$$
g(j)_k=\begin{cases}
(-1)^j\binom{n}{k}\Bigl[ \binom{k-1}{j-1}+\binom{n-k-1}{j-1}\Bigr],\qquad\text{if $j<k$},\\
   \\
(-1)^j\binom{n}{k}\Bigl[ \binom{n-k-1}{j-1}+\binom{n-k-1}{j-k}\Bigr],\qquad\text{if $j\ge k$}.\\
\end{cases}
$$
Then,
\begin{eqnarray*}
\sum_{j=0}^{\frac{n}{2}-1}\Bigl( \frac{1}{j+1}g(j)_k+\frac{1}{n-j}g(j+1)_k\Bigr)=
&\sum_{j=1}^{\frac{n}{2}-1}\frac{n+2}{(j+1)(n-j+1)}g(j)_k+\frac{2}{n+2}g({\frac{n}{2}})_k&\\
=&\binom{n}{k}(n+2)\Bigl[ \sum_{j=1}^{k-1}\frac{(-1)^j}{(j+1)(n-j+1)}\binom{k-1}{j-1}\Bigr]&(A)\\
+&\binom{n}{k}(n+2)\Bigl[\sum_{j=k}^{\frac{n}{2}-1}\frac{(-1)^j}{(j+1)(n-j+1)}\binom{n-k-1}{j-k}\Bigr]&(B)\\
+&\binom{n}{k}(n+2)\Bigl[\sum_{j=1}^{\frac{n}{2}-1}\frac{(-1)^j}{(j+1)(n-j+1)}\binom{n-k-1}{j-1}\Bigr]&(C)\\
+&\binom{n}{k}\frac{2}{n+2}(-1)^{\frac{n}{2}}\Bigl[ \binom{n-k-1}{\frac{n}{2}-1}+\binom{n-k-1}{\frac{n}{2}-k} \Bigr]&(D)
\end{eqnarray*}

Next, we modify these expressions as follows:
\begin{eqnarray*}
(A)
&=&\binom{n}{k}(n+2)\Bigl[ \sum_{j=1}^{k-1}(-1)^j\frac{j(n-j)(n-j-1)\cdots (k+1-j)(k-1)!}{(j+1)!(n-j+1)!}\Bigr]\\
&=&\binom{n}{k}(n+2)\Bigl[ \frac{(k-1)!}{(n+2)!}\sum_{j=1}^{k-1}(-1)^jP_k(j)\binom{n+2}{j+1}\Bigr]\\
&=&\binom{n}{k}\frac{(k-1)!}{(n+1)!}\Bigl[\sum_{j=0}^{n}(-1)^jP_k(j)\binom{n+2}{j+1}-P_k(k)\binom{n+2}{k+1}\Bigr],\\
\end{eqnarray*}
where $P_k(x)=x(n-x)(n-1-x)\cdots (k+1-x)$.
We now use a well known result from the theory of finite differences which asserts that, for any polynomial $P(x)$ of degree less than $n$,
\begin{equation}\label{ayay}
\sum_{j=0}^n (-1)^jP(j)\binom{n}{j}=0.
\end{equation}
This reduces $(A)$ to
$$
\binom{n}{k}\frac{(k-1)!}{(n+1)!}\Bigl[ P_k(-1)+P_k(n+1)-P_k(k)\binom{n+2}{k+1}\Bigr].
$$

 On the other hand:
\begin{eqnarray*}
(B)
&=&\binom{n}{k}(n+2)\Bigl[ (n-k-1)!\sum_{j={k}}^{\frac{n}{2}-1}(-1)^j\frac{(n-j)j(j-1)(j-2)\cdots (j-k+1)}{(n-j+1)!(j+1)!}\Bigr]\\
&=&\binom{n}{k}(n+2)\Bigl[ \frac{(n-k-1)!}{(n+2)!}\sum_{j={k}}^{\frac{n}{2}-1}(-1)^jQ_k(n-j)\binom{n+2}{j+1}\Bigr]\\
&=&\binom{n}{k}\frac{(n-k-1)!}{(n+1)!}\Bigl[ \sum_{j={\frac{n}{2}+1}}^{{n-k}}(-1)^jQ_k(j)\binom{n+2}{n-j+1}\Bigr],\\
\end{eqnarray*}
where $Q_k(x)=x(n-x)(n-1-x)\cdots (n-k+1-x)$, and in the last step we have replaced $n-j$ by $j$.
\begin{eqnarray*}
(C)
&=&\binom{n}{k}(n+2)\Bigl[ (n-k-1)!\sum_{j=1}^{\frac{n}{2}-1}(-1)^j\frac{j(n-j)(n-j-1)\cdots (n-j-k+1)}{(j+1)!(n-j+1)!}\Bigr]\\
&=&\binom{n}{k}(n+2)\Bigl[ \frac{(n-k-1)!}{(n+2)!}\sum_{j=1}^{\frac{n}{2}-1}(-1)^jQ_k(j)\binom{n+2}{n-j+1}\Bigr]\\
&=&\binom{n}{k}\frac{(n-k-1)!}{(n+1)!}\Bigl[ \sum_{j=1}^{\frac{n}{2}-1}(-1)^jQ_k(j)\binom{n+2}{n-j+1}+\sum_{j=n}^{n-k+1}(-1)^jQ_k(j)\binom{n+2}{n-j+1}\Bigr].
\end{eqnarray*}
Therefore,
\begin{eqnarray*}
(B)+(C)&=&\binom{n}{k}\frac{(n-k-1)!}{(n+1)!}\Bigl[ \sum_{j=0}^{n}(-1)^jQ_k(j)\binom{n+2}{n-j+1}-(-1)^{\frac{n}{2}}Q_k(\frac{n}{2})\binom{n+2}{\frac{n}{2}+1}\Bigr],\\
\end{eqnarray*}
which, in view of formula (\ref{ayay}), it becomes
$$\binom{n}{k}\frac{(n-k-1)!}{(n+1)!}\Bigl[Q_k(-1)+Q_k(n+1)-(-1)^{\frac{n}{2}}Q_k(\frac{n}{2})\binom{n+2}{\frac{n}{2}+1}\Bigr].
$$

Finally,
\begin{eqnarray*}
(D)&=&\frac{4}{(n+2)}(-1)^{\frac{n}{2}}\frac{(n-k-1)!}{(\frac{n}{2}-k)!(\frac{n}{2}-1)!}\frac{n!}{k!(n-k)!}.
\end{eqnarray*}
By simple evaluation we have:
\begin{align*}
P_k(-1)&=-\frac{(n+1)!}{(k+1)!};\quad
P_k(n+1)=(n+1)(n-k)!;\quad
P_k(k)=k(n-k)!;\\
Q_k(-1)&=-\frac{(n+1)!}{(n-k+1)!};\quad
Q_k(n+1)=(n+1)k!;\quad
Q_k(\frac{n}{2})=\frac{n}{2}\frac{(\frac{n}{2})!}{(\frac{n}{2}-k)!}.
\end{align*}
Then, a straightforward computation, substituting the above identities in the corresponding equations, reduces $(A)+(B)+(C)+(D)$ to
$$
\binom{n}{k}\Bigl[ -\frac{1}{(k+1)k}+\frac{1}{n\binom{n-1}{k-1}}-\frac{n+2}{(k+1)(n-k+1)} -\frac{1}{(n-k+1)(n-k)}+\frac{1}{n\binom{n-1}{k}}\Bigr].
$$
Observe that, on the one hand,
$$
\frac{1}{n\binom{n-1}{k}}+\frac{1}{n\binom{n-1}{k-1}}=\frac{n}{(n-k)k}\frac{1}{\binom{n}{k}}.
$$
On the other hand,
$$
-\frac{1}{(k+1)k}-\frac{n+2}{(k+1)(n-k+1)}-\frac{1}{(n-k+1)(n-k)}=-\frac{n}{(n-k)k}.
$$
Thus, we conclude that
\begin{eqnarray*}
\sum_{j=0}^{\frac{n}{2}-1}\Bigl( \frac{1}{j+1}g(j)_k+\frac{1}{n-j}g(j+1)_k\Bigr)&=&\binom{n}{k}\Bigl[\frac{n}{(n-k)k}\frac{1}{\binom{n}{k}}-\frac{n}{(n-k)k}\Bigr]\\
&=&\frac{n}{(n-k)k}\Bigl( 1-\binom{n}{k}\Bigr)=M_k,
\end{eqnarray*}
and the theorem is proved.
\end{proof}

We finish by presenting an unexpected symmetry relation between Euler and Miki's  identities.
\begin{theorem} \label{thmeuler} The vector of coefficients in the condensed form of the Euler's identity equals
$$-\frac{2}{n}\sum_{j=0}^{\frac{n}{2}-1}\Bigl( (n-j)g(j)+(j+1)g(j+1)\Bigr).$$
\end{theorem}
\begin{proof}
On the one hand, observe that the condensed form of the Euler's formula is
{$$
(n+1)B_n+\sum_{\substack{k=2\\ k\,\,\text{even}}}^{\frac{n}{2}-1}2\binom{n}{k}B_kB_{n-k}+\binom{n}{\frac{n}{2}}B_{\frac{n}{2}}^2=0.
$$}

 On the other hand, taking into account that $(n-j)\binom{n}{j}=(j+1)\binom{n}{j+1}$, a short computation yields
$$\sum_{j=0}^{\frac{n}{2}-1}\Bigl( (n-j)g(j)_0+(j+1)g(j+1)_0\Bigr)=-\frac{n}{2}(n+1).$$

Finally, another
straightforward computation gives
$$\sum_{j=0}^{\frac{n}{2}-1}\Bigl( (n-j)g(j)_k+(j+1)g(j+1)_k)\Bigr)=-{n}\binom{n}{k},$$
and the proof is complete.
\end{proof}

\begin{remark} Fix an even integer $n\ge 4$ and identify a given condensed Euler-type identity with the vector in $\bq^{{\lfloor\frac{n}{4}\rfloor}+1}$ of  coefficients of the given identity. Then,  theorems \ref{thmmiki} and \ref{thmeuler} tell us that both, the condensed Miki and Euler's
 identities, belong to the vector space $\langle \text{Im\,g}\rangle$ generated by the image
 $g\colon \Pi_{n-1}\to \mathbb{Q}^{{\lfloor\frac{n}{4}\rfloor }+1}$.

 Thus, one may ask whether any given condensed Euler-type identity lives in $\langle \text{Im\,g}\rangle$, that is, whether the latter equals the subspace $V$ of $\bq^{{\lfloor\frac{n}{4}\rfloor }+1}$ of all condensed Euler-type identities. The answer is negative as the inclusion   $\langle \text{Im\,g}\rangle\varsubsetneq V$ is proper. Indeed, $\text{dim} V= \lfloor\frac{n}{4}\rfloor$ and therefore, for $n=12$, $\text{dim}\, V=3$ while a direct computation shows that $\langle \text{Im\,g}\rangle$ has dimension $2$. Summarizing:
 $$\left\{ \begin{array}{c}
\text{Condensed Miki and} \\
\text{Euler's identities} \\
\end{array} \right\}\subset\langle \text{Im\,g}\rangle\varsubsetneq \left\{ \begin{array}{c}
\text{Condensed} \\
\text{Euler-type identities} \\
\end{array} \right\} \subset \mathbb{Q}^{{\lfloor\frac{n}{4}\rfloor }+1}.$$
\end{remark}

\bigskip
\bigskip

\noindent{\sc Institut de Recherche en Math\'ematique et Physique, Universit\'e Catholique de Louvain , 2 Chemin du Cyclotron B-1348, Louvain-la-Neuve, Belgium}.\hfill\break
{jose.carrasquel@uclouvain.be}
\bigskip

\noindent {\sc Departamento de \'Algebra, Geometr\'{\i}a y Topolog\'{\i}a, Universidad de M\'alaga, Ap.\ 59, 29080 M\'alaga, Spain}.\hfill\break {aniceto@uma.es}
\hfill\break
{ubuijs@uma.es}

\end{document}